\documentclass[12pt, a4paper]{article}
\usepackage{amsfonts}
\usepackage{mathrsfs}
\usepackage{latexsym}
\usepackage{xy}
\usepackage{amsfonts,amsmath,amssymb,amsthm}

\xyoption{all}

\newcommand{\bcen}{\begin{center}}     \newcommand{\ecen}{\end{center}}
\newcommand{\bay}{\begin{array}}      \newcommand{\eay}{\end{array}}
\newcommand{\beq}{\begin{eqnarray*}}      \newcommand{\eeq}{\end{eqnarray*}}

\def\gl{\mathrm{gl.dim}}

\def\Hom{\mathrm{Hom}}

\def\op{\mathrm{op}}

\def\Ext{\mathrm{Ext}}

\def\End{\mathrm{End}}

\def\Mod{\mathrm{Mod}}

\def\pd{\mathrm{pd}}

\def\proj{\mathrm{proj}}

\def\top{\mathrm{top}}
\def\tr{\mathrm{tr}}

\def\Tor{\mathrm{Tor}}

\begin{document}

\newtheorem{theorem}{Theorem}
\newtheorem{proposition}{Proposition}
\newtheorem{lemma}{Lemma}
\newtheorem{corollary}{Corollary}
\newtheorem{remark}{Remark}
\newtheorem{example}{Example}
\newtheorem{definition}{Definition}
\newtheorem*{conjecture}{Conjecture}
\newtheorem{question}{Question}

\title{\large\bf Hattori-Stallings trace and character}

\author{\large Yang Han}

\date{\footnotesize KLMM, ISS, AMSS,
Chinese Academy of Sciences, Beijing 100190, P.R. China.\\ E-mail:
hany@iss.ac.cn}

\maketitle

\bcen{\it Dedicated to the memory of Dieter Happel}\ecen

\begin{abstract} It is shown that Hattori-Stallings trace induces
a homomorphism of abelian groups, called Hattori-Stallings
character, from the $K_1$-group of endomorphisms of the perfect
derived category of an algebra to its zero-th Hochschild homology,
which provides a new proof of Igusa-Liu-Paquette Theorem, i.e., the
strong no loop conjecture for finite-dimensional elementary
algebras, on the level of complexes. Moreover, the Hattori-Stallings
traces of projective bimodules and one-sided projective bimodules
are studied, which provides another proof of Igusa-Liu-Paquette
Theorem on the level of modules.
\end{abstract}

\medskip


\medskip


\bigskip

\section{Introduction}

\indent\indent Global dimension is a quite important homological
invariant of an algebra or a ring. The (in)finiteness of global
dimension plays an important role in representation theory of
algebras. For instance, the bounded derived category of a
finite-dimensional algebra has Auslander-Reiten triangles if and
only if the algebra is of finite global dimension \cite{Hap2,Hap3}.
There are some well-known conjectures related to the (in)finiteness
of global dimension, such as no loop conjecture, Cartan determinant
conjecture --- the determinant of the Cartan matrix of an artin ring
of finite global dimension is 1 (ref. \cite{F}), Hochschild homology
dimension conjecture --- a finite-dimensional algebra is of finite
global dimension if and only if its Hochschild homology dimension is
0 (ref. \cite{Han}). To a finite-dimensional elementary algebra $A$,
we can associate a quiver $Q$, called its Gabriel quiver (ref.
\cite[Page 65]{ARS}). The (in)finiteness of the global dimension of
$A$ is closely related to the combinatorics of $Q$. If $Q$ has no
oriented cycles then $\gl A < \infty$ (ref. \cite{ENN}). Obviously,
its converse is not true in general. Nevertheless, if $\gl A <
\infty$ then $Q$ must have no loop, and 2-truncated cycle
\cite{BHM}. The former is due to the following conjecture:

\medskip

{\bf No loop conjecture.} {\it Let $A$ be an artin algebra of finite
global dimension. Then $\Ext^1_A(S,S)=0$ for every simple $A$-module
$S$.}

\medskip

The no loop conjecture was first explicitly established for artin
algebras of global dimension two \cite[Proposition]{GGZ}. For
finite-dimensional elementary algebras, which is just the case that
loop has its real geometric meaning, as shown in \cite{I}, this can
be easily derived from an earlier result of Lenzing \cite{L}. A
stronger version of no loop conjecture is the following:

\medskip

{\bf Strong no loop conjecture.} {\it Let $A$ be an artin algebra
and $S$ a simple $A$-module of finite projective dimension. Then
$\Ext^1_A(S,S)=0$.}

\medskip

The strong no loop conjecture is due to Zacharia \cite{I}, which is
also listed as a conjecture in Auslander-Reiten-Smal{\o}'s book
\cite[Page 410, Conjecture (7)]{ARS}. For finite-dimensional
elementary algebras, and particularly, for finite-dimensional
algebras over an algebraically closed field, it was proved in
\cite{ILP}. Some special cases were solved in
\cite{GSZ,J,LM,MP,Sk,Z}.

In this paper, we shall show that Hattori-Stallings trace induces a
homomorphism of abelian groups, called Hattori-Stallings character,
from the $K_1$-group of endomorphisms of the perfect derived
category of an algebra to its zero-th Hochschild homology (see
Section 2), which provides a neat proof of Igusa-Liu-Paquette
Theorem, i.e., the strong no loop conjecture for finite-dimensional
elementary algebras, on the level of complexes (see Section 3).
Moreover, in Section 4, we shall study the Hattori-Stallings traces
of projective bimodules and one-sided projective bimodules, which
provides a simpler proof of Igusa-Liu-Paquette Theorem on the level
of modules. A key point is the bimodule characterization of the
projective dimension of a simple module.

\section{Hattori-Stallings character}

\indent\indent In this section, we shall show that Hattori-Stallings
trace induces a homomorphism of abelian groups, called
Hattori-Stallings character, from the $K_1$-group of endomorphisms
of the perfect derived category of an algebra to its zero-th
Hochschild homology.

\subsection{Hattori-Stallings traces}

\indent\indent Let $A$ be a ring with identity. Denote by $\Mod A$
the category of right $A$-modules, and by $\proj A$ the full
subcategory of $\Mod A$ consisting of all finitely generated
projective right $A$-modules. Denote by $D(A)$ the unbounded derived
categories of the complexes of right $A$-modules, and by $K^b(\proj
A)$ the homotopy category of the bounded complexes of finitely
generated projective right $A$-modules, which is triangle equivalent
to the perfect derived category of $A$.

For each $P \in \proj A$, there is an isomorphism of abelian groups
$$\phi_P : P \otimes_A \Hom_A(P,A) \rightarrow \End_A(P)$$ defined by
$\phi_P(p \otimes f)(p')=pf(p')$ for all $p,p' \in P$ and $f \in
\Hom_A(P,A)$. There is also a homomorphism of abelian groups
$$\psi_P : P \otimes_A \Hom_A(P,A) \rightarrow A/[A,A]$$
defined by $\psi_P(p \otimes f)= \overline{f(p)}$ for all $p \in P$
and $f \in \Hom_A(P,A)$. Here, $[A,A]$ is the additive subgroup of
$A$ generated by all commutators $[a,b] := ab-ba$ with $a,b \in A$.
It is well-known that the abelian group $A/[A,A]$ is isomorphic to
the zero-th Hochschild homology group $HH_0(A)$ of $A$. The
homomorphism of abelian groups
$$\tr_P := \psi_P\phi^{-1}_P : \End_A(P) \rightarrow A/[A,A]$$ is
called {\it the Hattori-Stallings trace of $P$}.

Hattori-Stallings trace has the following properties:

\begin{proposition} {\rm (Hattori \cite{Hat}, Stallings \cite{S}, Lenzing \cite{L})} Let $P,P',P'' \in \proj A$.

{\rm (HS1)} If $f \in \End_A(P)$ and $g \in \Hom_A(P,P')$ is an
isomorphism then $\tr_P(f) = \tr_{P'}(gfg^{-1})$.

{\rm (HS2)} If $f,f' \in \End_A(P)$ then $\tr_P(f+f') = \tr_P(f) +
\tr_P(f')$.

{\rm (HS3)} If $f= \left[\begin{array}{cc} f_{11} & f_{12} \\ f_{21}
& f_{22} \end{array}\right] \in \End_A(P \oplus P')$ then
$\tr_{P\oplus P'}(f) = \tr_P(f_{11}) + \tr_{P'}(f_{22})$.

{\rm (HS4)} If $f \in \Hom_A(P,P')$ and $g \in \Hom_A(P',P)$ then
$\tr_P(gf) = \tr_{P'}(fg)$.

{\rm (HS5)} If $$\xymatrix{  0 \ar[r] & P' \ar[d]^{f'} \ar[r] & P
\ar[d]^{f} \ar[r] & P'' \ar[d]^{f''} \ar[r] & 0 \\ 0 \ar[r] & P'
\ar[r] & P \ar[r] & P'' \ar[r] & 0 }$$ is a commutative diagram with
exact rows then $\tr_P(f) = \tr_{P'}(f') + \tr_{P''}(f'')$.

{\rm (HS6)} If $l_a \in \End_A(A)$ is the left multiplication by $a
\in A$ then $\tr_A(l_a)=\bar{a}$, the equivalence class of $a$ in
$A/[A,A]$.
\end{proposition}

\subsection{$K_1$-groups of endomorphisms}

\indent\indent Let ${\cal C}$ be a category. Denote by $\mathrm{end}
{\cal C}$ the {\it category of endomorphisms} of ${\cal C}$, whose
objects are all pairs $(C,f)$ with $C \in {\cal C}$ and $f \in
\End_{{\cal C}}(C)$ and whose Hom sets are $\Hom_{\mathrm{end} {\cal
C}}((C,f),(C',f')) := \{g \in \Hom_{{\cal C}}(C,C') | gf=f'g \}$.
Obviously, if ${\cal C}$ is a skeletally small category then so is
$\mathrm{end} {\cal C}$.

For a skeletally small triangulated category ${\cal T}$, we define
its {\it $K_1$-group of endomorphisms} (cf. \cite[Chapter III]{B}),
denoted by $K_1(\mathrm{end} {\cal T})$, to be the factor group of
the free abelian group generated by all isomorphism classes of
objects in $\mathrm{end} {\cal T}$ modulo the relations:

(K1) $[(T,f+f')]=[(T,f)]+[(T,f')]$ for all $T \in {\cal T}$ and
$f,f' \in \End_{{\cal T}}(T)$.

(K2) $[(T,f)]=[(T',f')]+[(T'',f'')]$ for every commutative diagram
$$\xymatrix{ T' \ar[d]^{f'} \ar[r] & T
\ar[d]^{f} \ar[r] & T'' \ar[d]^{f''} \ar[r] & \\ T' \ar[r] & T
\ar[r] & T'' \ar[r] & }$$ with triangles as rows.

Clearly, if two skeletally small triangulated categories are
triangle equivalent then their $K_1$-groups of endomorphisms are
isomorphic.

\subsection{Hattori-Stallings character}

\indent\indent For any ring $A$ with identity, both the exact
category $\proj A$ and the triangulated category $K^b(\proj A)$ are
skeletally small. So is $\mathrm{end} K^b(\proj A)$.

The main result in this section is the following:

\begin{theorem} Let $A$ be a ring with identity. Then the map
$$\tr : K_1(\mathrm{end} K^b(\proj A)) \rightarrow A/[A,A], \;\;\;\; [(P^{\bullet},\overline{f^{\bullet}})]
\mapsto \sum_{i \in \mathbb{Z}}(-1)^i \tr_{P^i}(f^i), $$ is a
homomorphism of abelian groups, called {\rm the Hattori-Stallings
character of $A$}, which satisfies the trace property {\rm (TP):}
$$\tr([(P^{\bullet},\overline{g^{\bullet}} \circ
\overline{f^{\bullet}})]) =
\tr([(P'^{\bullet},\overline{f^{\bullet}} \circ
\overline{g^{\bullet}})])$$ for all $\overline{f^{\bullet}} \in
\Hom_{K^b(\proj A)}(P^{\bullet},P'^{\bullet})$ and
$\overline{g^{\bullet}} \in \Hom_{K^b(\proj
A)}(P'^{\bullet},P^{\bullet})$.
\end{theorem}

\begin{proof} Since $P^{\bullet} \in K^b(\proj A)$, $P^i$ is zero for almost all $i \in \mathbb{Z}$.
Thus the sum $\sum_{i \in \mathbb{Z}}(-1)^i \tr_{P^i}(f^i)$ makes sense.

\medskip

{\it Step 1.} $\sum_{i \in \mathbb{Z}}(-1)^i \tr_{P^i}(f^i)$ is
independent of the choice of the representative $f^{\bullet}$ of the
homotopy equivalence class $\overline{f^{\bullet}}$. Indeed, if
$\overline{f^{\bullet}} = \overline{f'^{\bullet}}$, then
$f^{\bullet} - f'^{\bullet} =
s^{\bullet+1}d^{\bullet}+d^{\bullet-1}s^{\bullet}$ for some homotopy
map $s^{\bullet}$, where $d^{\bullet}$ is the differential of
$P^{\bullet}$. It follows from (HS2) and (HS4) that
$$\begin{array}{rl} & \sum_{i \in \mathbb{Z}}(-1)^i \tr_{P^i}(f^i) - \sum_{i \in
\mathbb{Z}}(-1)^i \tr_{P^i}(f'^i) \\ = & \sum_{i \in
\mathbb{Z}}(-1)^i \tr_{P^i}(f^i-f'^i) \\ = & \sum_{i \in
\mathbb{Z}}(-1)^i \tr_{P^i}(s^{i+1}d^i+d^{i-1}s^i) \\ = & \sum_{i
\in
\mathbb{Z}}(-1)^i (\tr_{P^i}(s^{i+1}d^i)+\tr_{P^i}(d^{i-1}s^i)) \\
= & \sum_{i \in \mathbb{Z}}(-1)^i
(\tr_{P^i}(s^{i+1}d^i)+\tr_{P^{i-1}}(s^id^{i-1})) = 0. \end{array}$$

{\it Step 2.}  $\sum_{i \in \mathbb{Z}}(-1)^i \tr_{P^i}(f^i)$ is
also independent of the choice of the representative
$(P^{\bullet},\overline{f^{\bullet}})$ of the isomorphism class
$[(P^{\bullet},\overline{f^{\bullet}})]$. Indeed, if
$(P^{\bullet},\overline{f^{\bullet}}) \cong
(P'^{\bullet},\overline{f'^{\bullet}})$ then there are morphisms
$\overline{g^{\bullet}} \in \Hom_{K^b(\proj
A)}(P^{\bullet},P'^{\bullet})$ and $\overline{g'^{\bullet}} \in
\Hom_{K^b(\proj A)}(P'^{\bullet},P^{\bullet})$ such that
$\overline{g'^{\bullet}} \circ \overline{g^{\bullet}} = \bar{1}$,
$\overline{g^{\bullet}} \circ \overline{g'^{\bullet}} = \bar{1}$,
and $\overline{f'^{\bullet}} \circ \overline{g^{\bullet}} =
\overline{g^{\bullet}} \circ \overline{f^{\bullet}}$. Thus
$\overline{f^{\bullet}} = \overline{g'^{\bullet}} \circ
\overline{g^{\bullet}} \circ \overline{f^{\bullet}}$ and
$\overline{g^{\bullet}} \circ \overline{f^{\bullet}} \circ
\overline{g'^{\bullet}} = \overline{f'^{\bullet}} \circ \overline{g^{\bullet}} \circ
\overline{g'^{\bullet}} = \overline{f'^{\bullet}}$. It follows from
Step 1 and (HS4) that
$$\begin{array}{rl} \sum_{i \in \mathbb{Z}}(-1)^i \tr_{P^i}(f^i) & =
\sum_{i \in \mathbb{Z}}(-1)^i \tr_{P^i}(g'^ig^if^i) \\ & = \sum_{i
\in \mathbb{Z}}(-1)^i \tr_{P'^i}(g^if^ig'^i) \\ & = \sum_{i \in
\mathbb{Z}}(-1)^i \tr_{P'^i}(f'^ig^ig'^i) \\ & = \sum_{i \in
\mathbb{Z}}(-1)^i \tr_{P'^i}(f'^i). \end{array}$$

Now we have shown that $\tr$ is well-defined on the free abelian
group generated by the isomorphism classes of $\mathrm{end}
K^b(\proj A)$.

\medskip

{\it Step 3.}
$\tr([(P^{\bullet},\overline{f^{\bullet}}+\overline{f'^{\bullet}})])
=\tr([(P^{\bullet},\overline{f^{\bullet}})])
+\tr([(P^{\bullet},\overline{f'^{\bullet}})])$ for all $P^{\bullet}
\in K^b(\proj A)$ and $\overline{f^{\bullet}},
\overline{f'^{\bullet}} \in \End_{K^b(\proj A)}(P^{\bullet})$.
Indeed, this is clear by (HS2).

\medskip

{\it Step 4.} $\tr([(P^{\bullet},\overline{f^{\bullet}})]) =
\tr([(P'^{\bullet},\overline{f'^{\bullet}})])+\tr([(P''^{\bullet},\overline{f''^{\bullet}})])$
for every commutative diagram $$\xymatrix{ P'^{\bullet}
\ar[d]^{\overline{f'^{\bullet}}} \ar[r]^{\overline{u^{\bullet}}} &
P^{\bullet} \ar[d]^{\overline{f^{\bullet}}}
\ar[r]^{\overline{v^{\bullet}}} & P''^{\bullet}
\ar[d]^{\overline{f''^{\bullet}}} \ar[r] &
\\ P'^{\bullet} \ar[r]^{\overline{u^{\bullet}}} & P^{\bullet} \ar[r]^{\overline{v^{\bullet}}} & P''^{\bullet} \ar[r] & }$$
with triangles as rows. Indeed, in $K^b(\proj A)$ each triangle
$P'^{\bullet} \stackrel{\overline{u^{\bullet}}}{\longrightarrow}
P^{\bullet} \stackrel{\overline{v^{\bullet}}}{\longrightarrow}
P''^{\bullet} \longrightarrow$ is isomorphic to a triangle
$$\xymatrix@!=4pc{P'^{\bullet} \ar[r]^-{\left[\!\!\!\begin{array}{c} \bar{1}
\\ 0 \end{array}\!\!\!\right]} & \mathrm{Cyl}(u^{\bullet})
\ar[r]^-{\left[\!\!\!\begin{array}{cc} 0 & \bar{1}
\end{array}\!\!\!\right]} & \mathrm{Cone}(u^{\bullet})
\ar[r] &}$$ where $\mathrm{Cyl}(u^{\bullet})$ and
$\mathrm{Cone}(u^{\bullet})$ are the cylinder and cone of the
cochain map $u^{\bullet} : P'^{\bullet} \longrightarrow P^{\bullet}$
respectively. Thus, by Step 2, it is enough to consider the case that the following diagram
$$\xymatrix@!=4pc{P'^{\bullet} \ar[d]^-{\overline{f'^{\bullet}}} \ar[r]^-{\left[\!\!\!\begin{array}{c} \bar{1}
\\ 0 \end{array}\!\!\!\right]} & P'^{\bullet} \oplus P''^{\bullet} \ar[d]^-{\left[\!\!\!\begin{array}{cc}
\overline{f'^{\bullet}} & \overline{f'''^{\bullet}}
\\ 0 & \overline{f''^{\bullet}} \end{array}\!\!\!\right]}
\ar[r]^-{\left[\!\!\!\begin{array}{cc} 0 & \bar{1}
\end{array}\!\!\!\right]} & P''^{\bullet} \ar[d]^-{\overline{f''^{\bullet}}}
\ar[r] & \\ P'^{\bullet} \ar[r]^-{\left[\!\!\!\begin{array}{c}
\bar{1}
\\ 0 \end{array}\!\!\!\right]} & P'^{\bullet} \oplus P''^{\bullet}
\ar[r]^-{\left[\!\!\!\begin{array}{cc} 0 & \bar{1}
\end{array}\!\!\!\right]} & P''^{\bullet}
\ar[r] & }$$ with triangles as rows is commutative. In this case, by (HS3), we have $\sum_{i \in
\mathbb{Z}}(-1)^i \tr_{P'^i}(f'^i) + \sum_{i \in \mathbb{Z}}(-1)^i
\tr_{P''^i}(f''^i) = \sum_{i \in \mathbb{Z}}(-1)^i \tr_{P'^i \oplus
P''^i}(\left[\!\!\!\begin{array}{cc} f'^i & f'''^i
\\ 0 & f''^i \end{array}\!\!\!\right])$.

Now we have shown that the Hattori-Stallings character $\tr$ is
well-defined. Next, we prove that it satisfies trace property (TP).

\medskip

{\it Step 5.} $\tr([(P^{\bullet},\overline{g^{\bullet}} \circ
\overline{f^{\bullet}})]) =
\tr([(P'^{\bullet},\overline{f^{\bullet}} \circ
\overline{g^{\bullet}})])$ for all morphisms $\overline{f^{\bullet}}
\in \Hom_{K^b(\proj A)}(P^{\bullet},P'^{\bullet})$ and
$\overline{g^{\bullet}} \in \Hom_{K^b(\proj
A)}(P'^{\bullet},P^{\bullet})$. Indeed, this is clear by Step 2 and
(HS4).
\end{proof}

\section{Igusa-Liu-Paquette Theorem}

\indent\indent In this section, we shall apply Hattori-Stallings
character to give a new proof of Igusa-Liu-Paquette Theorem on the
level of complexes. From now on, let $k$ be a field and $A$ a
finite-dimensional elementary $k$-algebra, i.e., $A/J \cong k^n$ for
some natural number $n$, where $J$ denotes the Jacobson radical of
$A$.

\subsection{Projective dimension}

\indent\indent Some homological properties on modules can be
characterized by those of bimodules. For instance, Happel showed
that for a finite-dimensional $k$-algebra $A$, $\gl A =\pd_{A^e}A$
(ref. \cite{Hap1}). In this subsection, we shall give a bimodule
characterization of the projective dimension of a simple module. For
this, we need the following well-known result, which implies that
$\top A = A/J$ is a ``testing module" of the projective dimension of
an $A$-module:

\begin{lemma} \label{test} Let $A$ be an artin algebra, and $M \neq 0$ a
finitely generated left $A$-module. Then $\pd_AM = \mbox{\rm
sup}\{i|\Ext^i_A(M,A/J) \neq 0 \} = \mbox{\rm sup}\{i|\Tor^A_i(A/J,
M) \neq 0 \}.$
\end{lemma}

\begin{proof} Let $P^\bullet$ be a minimal projective resolution of the left $A$-module $M$.
Then all the differentials of the complex $\Hom_A(P^\bullet,A/J)$
are zero. Thus $\Ext^i_A(M,A/J)=\Hom_A(P^{-i},A/J)$. Hence, $\pd_AM
= \mbox{sup}\{i|P^{-i} \neq 0 \} = \mbox{sup}\{i|\Hom_A(P^{-i},A/J)
\neq 0 \} = \mbox{sup}\{i|\Ext^i_A(M,A/J) \neq 0 \}$.

Similarly, all the differentials of the complex $A/J \otimes_A
P^\bullet$ are zero. Thus $\Tor^A_i(A/J,M)= A/J \otimes_A P^{-i}$.
Therefore, $\pd_AM = \mbox{sup}\{i|P^{-i} \neq 0 \} =
\mbox{sup}\{i|A/J \otimes_A P^{-i} \neq 0 \} =
\mbox{sup}\{i|\Tor^A_i(A/J,M) \neq 0 \}$.
\end{proof}

A key point of this paper is the following observation:

\begin{lemma} \label{1 to 2 sided} Let $A$ be a finite-dimensional elementary $k$-algebra, $S=Ae/Je$
the left simple $A$-module corresponding to a primitive idempotent
$e$ in $A$, and $\bar{A} := A/A(1-e)A$. Then $\pd_A S = \pd_{A
\otimes_k \bar{A}^{\op}} \bar{A}$.

\end{lemma}

\begin{proof} We have isomorphisms $\Tor^A_i(A/J,S) \cong H^{-i}(A/J \otimes^L_AS) \cong
H^{-i}(A/J \otimes^L_A \bar{A} \otimes^L_{\bar{A}} S) \cong
H^{-i}((A/J \otimes_k S) \otimes^L_{A \otimes_k \bar{A}^{\op}}
{\bar{A}}) \cong \Tor^{A \otimes_k \bar{A}^{\op}}_i(A/J \otimes_k
S,\bar{A})$. Applying Lemma~\ref{test} twice, we obtain $\pd_A S =
\pd_{A \otimes_k \bar{A}^{\op}} \bar{A}$, since $A/J \otimes_k S =
\top (A \otimes_k \bar{A}^{\op})$.
\end{proof}

\subsection{A new proof of Igusa-Liu-Paquette Theorem}

\begin{theorem} \label{ILP} {\rm (Igusa-Liu-Paquette \cite{ILP})} Let
$A$ be a finite-dimensional elementary $k$-algebra, $S$ a left
simple $A$-module, and $\pd_AS < \infty$. Then $\Ext^1_A(S,S)=0$.
\end{theorem}

\begin{proof} We may assume that $S$ is
the left simple $A$-module $Ae/Je$ corresponding to a primitive
idempotent $e$ in $A$ and $\bar{A} := A/A(1-e)A$.

We have the following commutative diagram in $D(A)$:
$$\xymatrix{ J^{j+1} \ar[d]^-{l_a} \ar[r] &
J^j \ar[d]^-{l_a} \ar[r] & J^j/J^{j+1} \ar[d]^-{l_a=0} \ar[r] &\\
J^{j+1} \ar[r] & J^j \ar[r] & J^j/J^{j+1} \ar[r] &}$$ with triangles
as rows for all $a \in J$, the Jacobson radical of $A$, and $0 \leq
j \leq t-1$ where $t$ is the Loewy length of $A$. Applying the
derived tensor functor $- \otimes^L_A \bar{A}$ to the commutative
diagram above, we obtain the following commutative diagram in
$D(\bar{A})$:
$$\xymatrix{ J^{j+1} \otimes^L_A \bar{A} \ar[d]^-{l_a \otimes^L_A \bar{A}}
\ar[r] &
J^j \otimes^L_A \bar{A} \ar[d]^-{l_a \otimes^L_A \bar{A}}
\ar[r] &
(J^j/J^{j+1}) \otimes^L_A \bar{A} \ar[d]^-{l_a \otimes^L_A \bar{A}=0} \ar[r] &\\
J^{j+1} \otimes^L_A \bar{A} \ar[r] & J^j \otimes^L_A \bar{A} \ar[r]
& (J^j/J^{j+1}) \otimes^L_A \bar{A} \ar[r] &}$$ with triangles as
rows for all $a \in J$ and $0 \leq j \leq t-1$. By the assumption
$\pd_AS< \infty$ and Lemma~\ref{1 to 2 sided}, we have a bounded
finitely generated projective $A$-$\bar{A}$-bimodules resolution
$P^\bullet$ of $\bar{A}$. Thus we have the following commutative
diagram in $K^b(\proj \bar{A})$:
$$\xymatrix{ J^{j+1} \otimes_A P^\bullet \ar[r] \ar[d]^{l_a} & J^j \otimes_A P^\bullet
\ar[r] \ar[d]^{l_a} &
(J^j/J^{j+1}) \otimes_A P^\bullet \ar[d]^{l_a=0} \ar[r] & \\
J^{j+1} \otimes_A P^\bullet \ar[r] & J^j \otimes_A P^\bullet \ar[r]
& (J^j/J^{j+1}) \otimes_A P^\bullet \ar[r] & }$$ with triangles as
rows for all $a \in J$ and $0 \leq j \leq t-1$. Therefore, for any
$\bar{a} \in \bar{J}$, the Jacobson radical of $\bar{A}$, the
equivalence class of $\bar{a}$ in $\bar{A}/[\bar{A},\bar{A}]$

$$\begin{array}{rcl} \bar{\bar{a}} = \tr([(\bar{A},l_{\bar{a}})]) = \tr([(\bar{A},l_a)]) &
= & \tr([(J^0 \otimes_A P^\bullet, l_a)]) \\ & = & \tr([(J^1
\otimes_A P^\bullet, l_a)]) \\ & = & \cdots
\\ & = & \tr([(J^t \otimes_A P^\bullet, l_a)]) \\ & = & \tr([(0,0)]) = 0.
\end{array}$$ Hence, $\bar{J} \subseteq [\bar{A},\bar{A}]$.

Let $A' := \bar{A}/\bar{J}^2$ and $J' = \bar{J}/\bar{J}^2$ its
Jacobson radical. Then $A'$ is a local algebra with radical square
zero, and thus commutative. Since $\bar{J} \subseteq [\bar{A},
\bar{A}]$, we have $J' \subseteq [A',A'] = 0$, i.e., $J'=0$. Hence,
$\Ext^1_A(S,S) \cong eJe/eJ^2e \cong J' = 0$.

\end{proof}

\section{Hattori-Stallings traces of bimodules}

\indent\indent In this section, we shall study the Hattori-Stallings
traces of projective bimodules and one-sided projective bimodules,
which provides another proof of Igusa-Liu-Paquette Theorem on the
level of modules.

Firstly, we consider the Hattori-Stallings traces of finitely
generated projective bimodules.

\begin{proposition} \label{2 proj bimod trace} Let $A$ and $B$ be
finite-dimensional $k$-algebras, and $P$ a finitely generated
projective $A$-$B$-bimodule. Then $\tr_{P_B}(l_a)=0$ for all $a \in
J$, the Jacobson radical of $A$.
\end{proposition}

\begin{proof} We have the following commutative diagram in $\Mod A$:
$$\xymatrix{ 0 \ar[r] & J^{j+1} \ar[d]^-{l_a} \ar[r] &
J^j \ar[d]^-{l_a} \ar[r] & J^j/J^{j+1} \ar[d]^-{l_a=0} \ar[r] & 0 \\
0 \ar[r] & J^{j+1} \ar[r] & J^j \ar[r] & J^j/J^{j+1} \ar[r] & 0}$$
with exact rows for all $a \in J$ and $0 \leq j \leq t-1$ where $t$
is the Loewy length of $A$. Since $P$ is a finitely generated
projective $A$-$B$-bimodule, we have the following commutative
diagram in $\proj B$:
$$\xymatrix{ 0 \ar[r] & J^{j+1} \otimes_A P \ar[d]^-{l_a} \ar[r] &
J^j \otimes_A P \ar[d]^-{l_a} \ar[r] & (J^j/J^{j+1}) \otimes_A P \ar[d]^-{l_a=0} \ar[r] & 0 \\
0 \ar[r] & J^{j+1} \otimes_A P \ar[r] & J^j \otimes_A P \ar[r] &
(J^j/J^{j+1}) \otimes_A P \ar[r] & 0}$$ with exact rows for all $a
\in J$ and $0 \leq j \leq t-1$. It follows from (HS5) that
$\tr_{P_B}(l_a) = \tr_{J^0 \otimes_A P_B}(l_a) = \tr_{J^1 \otimes_A
P_B}(l_a) = \cdots = \tr_{J^{t-1} \otimes_A P_B}(l_a=0) = 0$ for all
$a \in J$.
\end{proof}

Secondly, we consider the Hattori-Stallings traces of finitely
generated one-sided projective bimodules.

\begin{proposition} \label{1 proj bimod trace} Let $A$ and $B$ be
finite-dimensional $k$-algebras, $M$ a finitely generated
$A$-$B$-bimodule which is projective as a right $B$-module, and
$P^\bullet$ a finitely generated projective $A$-$B$-bimodule
resolution of $M$. Then
$$\tr_{M_B}(l_a) = (-1)^i \tr_{\Omega_i(M)}(l_a)$$ for all $a \in J$
and $i \in \mathbb{N}$, where $\Omega_i(M)$ is the $i$-th syzygy of
$M$ on $P^\bullet$.
\end{proposition}

\begin{proof} Since $M_B$ is projective, all $\Omega_i(M)_B$'s are projective.
We have the following commutative diagrams in $\proj B$:
$$\xymatrix{ 0 \ar[r] & \Omega_i(M) \ar[d]^-{l_a} \ar[r] &
P^{-i+1} \ar[d]^-{l_a} \ar[r] & \Omega_{i-1}(M) \ar[d]^-{l_a} \ar[r] & 0 \\
0 \ar[r] & \Omega_i(M) \ar[r] & P^{-i+1} \ar[r] & \Omega_{i-1}(M)
\ar[r] & 0}$$ with exact rows for all $a \in J$ and $i \geq 1$. By
Proposition~\ref{2 proj bimod trace} and (HS5), we obtain
$\tr_{\Omega_i(M)}(l_a) = - \tr_{\Omega_{i-1}(M)}(l_a)$, thus
$\tr_{M_B}(l_a) = \tr_{\Omega_0(M)}(l_a) = - \tr_{\Omega_1(M)}(l_a)
\linebreak  = \cdots = (-1)^i \tr_{\Omega_i(M)}(l_a)$ for all $a \in
J$ and $i \in \mathbb{N}$.
\end{proof}

Finally, we provide another proof of Igusa-Liu-Paquette Theorem,
i.e., Theorem~\ref{ILP}, on the level of modules.

\begin{proof} By the assumption $\pd_AS < \infty$ and Lemma~\ref{1 to 2 sided},
we have $\pd_{A \otimes_k \bar{A}^{\op}} \bar{A} < \infty$. It
follows from Proposition~\ref{1 proj bimod trace} that, for any
$\bar{a} \in \bar{J}$, the equivalence class of $\bar{a}$ in
$\bar{A}/[\bar{A},\bar{A}]$, $\bar{\bar{a}} =
\tr_{\bar{A}}(l_{\bar{a}}) = \tr_{\bar{A}}(l_a) = (-1)^i
\tr_{\Omega_i(M)}(l_a)$ which equals 0 for $i> \pd_{A \otimes_k
\bar{A}^{\op}} \bar{A}$. Thus $\bar{J} \subseteq [\bar{A},\bar{A}]$.
Then we may continue as the last paragraph of the proof of
Theorem~\ref{ILP} in Section 3.2.
\end{proof}

\noindent {\footnotesize {\bf ACKNOWLEDGMENT.} The author is
sponsored by Project 11171325 NSFC.}

\end{document}